\documentclass[11pt,letterpaper, reqno]{amsart}

\renewcommand{\Im}{\operatorname{Im}}

\newcommand{\defeq}{\stackrel{\rm{def}}{=}}

\usepackage{amssymb}
\usepackage{amsthm}
\usepackage{amsxtra}
\usepackage{graphicx}
\usepackage{color}

\newcommand{\R}{\mathbb R}
\newcommand{\C}{\mathbb C}

\newtheorem{theorem}{Theorem}[section]
\newtheorem{definition}[theorem]{Definition}
\newtheorem{proposition}{Proposition}[section]
\newtheorem{lemma}[proposition]{Lemma}
\newtheorem{corollary}[proposition]{Corollary}

\theoremstyle{remark}

\numberwithin{equation}{section}

\begin{document}
	\title[Well-posedness and blow-up in generalized Hartree]
	{On well-posedness and blow-up\\ 
	in the generalized Hartree equation} 
	
	\author[Anudeep K. Arora]{Anudeep Kumar Arora}
	\address{Department of Mathematics  \& Statistics\\
		Florida International University,  Miami, FL, USA}
	\curraddr{}
	\email{simplyandy7@gmail.com}
	\thanks{} 
	
	\author[Svetlana Roudenko]{Svetlana Roudenko}
	\address{Department of Mathematics \& Statistics\\
		Florida International University,  Miami, FL, USA}
	\curraddr{}
	\email{sroudenko@fiu.edu}
	\thanks{}
	
	\subjclass[2010]{Primary: 35Q55, 35Q40; secondary:  37K05} 
	
	\keywords{Hartree equation, Choquard-Pekar equation, convolution nonlinearity, global well-posedness, blow-up criteria}
	
	\date{} 
	
\begin{abstract}
We study the generalized Hartree equation, which is a nonlinear Schr\"odinger-type equation with a nonlocal potential $iu_t + \Delta u  + (|x|^{-b} \ast |u|^p)|u|^{p-2}u=0, x \in \mathbb{R}^N$.
We establish the local well-posedness at the non-conserved critical regularity $\dot{H}^{s_c}$ for $s_c \geq 0$, which also includes the energy-supercritical regime $s_c>1$ 
(thus, complementing the work in [3], where the authors obtained the $H^1$ well-posedness in the intercritical regime together with classification of solutions under the mass-energy threshold). We next extend the local theory to global: for small data we obtain global in time existence and for initial data with positive energy and certain size of variance we show the finite time blow-up (blow-up criterion). Both of these results hold regardless of the criticality of the equation. 
In the intercritical setting the criterion produces blow-up solutions with the initial values above the mass-energy threshold. 
We conclude with examples showing currently known thresholds for global vs. finite time behavior. 
\end{abstract}

\maketitle

\section{Introduction}
We consider the focusing generalized Hartree (gH) equation of the form 
\begin{align}\label{gH}
iu_t+\Delta u+\left(\frac1{|x|^{b}}\ast|u|^p\right)|u|^{p-2}u=0,\quad (x,t) \in \R^N \times \R 
\end{align}
for $p\geq 2$ and $0<b<N$. 
The equation \eqref{gH} is a generalization of the standard Hartree equation with $p=2$, which arises, for example, as an effective evolution equation in the mean-field limit of many-body quantum systems, see \cite{H74}, \cite{GV80}, \cite{Sp80}, \cite{FL04},\cite{FTY2000}; in the Chandrasekhar theory of stellar collapse \cite{LY87}; as an electrostatic version of the Maxwell-Schr\"odinger system \cite{CG04}, \cite{Lieb03}, in Bose-Einstein condensates of a gas of bosonic particles with long-range dipole-dipole interactions, see \cite{LMSLP09}, \cite{L-02}, \cite{L-10} and in various other phenomena.

The equation \eqref{gH} enjoys several invariances, among them the scaling invariance: if $u(x,t)$ solves \eqref{gH}, then so does
\begin{equation}\label{E:scale}
u_{\lambda}(x,t) = \lambda^{\frac{N-b+2}{2(p-1)}}u(\lambda x,\lambda^2 t). 
\end{equation}
This implies that $\dot{H}^{s_c}$ norm is invariant under the above scaling provided the critical scaling index $s_c$ is
\begin{equation}\label{E:scaling}
\displaystyle{s_c=\frac{N}{2}-\frac{N-b+2}{2(p-1)}}.
\end{equation}
The equation \eqref{gH} is called $\dot{H}^{s_c}$-critical if for given $b,N,p$ in \eqref{gH} the $\dot{H}^{s_c}$ norm is scale-invariant with scaling \eqref{E:scale} and $s_c$ defined by \eqref{E:scaling}.

During their lifespan, solutions to \eqref{gH} satisfy mass conservation 
\begin{equation}\label{mass}
	M[u(t)]\defeq\int_{\R^N}^{}|u|^2\,dx = M[u_0],
\end{equation}
energy conservation
\begin{equation}\label{energy}
	E[u(t)] \defeq \frac{1}{2}\int_{\R^N}^{}|\nabla u|^2- \frac{1}{2p}Z(u)  = E[u_0],
\end{equation}
where $Z(u) \defeq \int_{\R^N}(|x|^{-b}\ast|u|^p)|u|^{p}\,dx$, and momentum conservation  
\begin{equation}\label{moment}
P[u(t)]\defeq\Im\left(\int_{\R^N}^{}\bar{u}\,\nabla u\,dx\right) = P[u_0].
\end{equation}

In this paper we are interested in understanding the long-term behavior of solutions, either global in time or finite time existence, for a variety of criticality cases ($s_c\geq 0$) of the gHartree equation \eqref{gH}. The local well-posedness is the starting point, and in this note we obtain the local well-posedness at the critical regularity  $\dot{H}^{s_c}$, $s_c \geq 0$, which is not necessarily conserved (or even bounded in the focusing case). The local existence is then extended to the global existence for small $\dot{H}^{s_c}$ data. On the other hand, we show that large data may blow-up in finite time. For that we give a sufficient condition for blow-up and show examples of Gaussian data with thresholds in various (energy-subcritical, critical and supercritical) cases. Such examples are important for studying the actual dynamics of finite time blow-up. For example, in \cite{YRZ3} the dynamics of stable blow-up is investigated (including rates and profiles) for the generalized Hartree in the mass-critical and supercritical regimes, and is compared with known blow-up dynamics of the (local) nonlinear Schr\"odinger equation.

In \cite{AKAR1} we showed that the Cauchy problem for the equation \eqref{gH} with the initial data $u(x,0)=u_0(x)$ is locally well-posed in $H^1$ provided $s_c<1$ (note that the nonlinearity for the $H^1$ local well-posedness in \cite{AKAR1} is always $H^1$-subcritical). 
Our first result in this paper addresses local in time solutions at the $\dot{H}^{s_c}$ regularity. 
\begin{theorem}\label{lwp}
Let $0<b<N$, $N \geq 1$ and $p\geq 2$ so that $s_c\geq 0$.  
Assume in addition that if $p$ is not an even integer, then $s_c< p-1$. 
Let $u_0\in \dot{H}^{s_c}(\R^N)$.  
Then there exists a unique solution $u(x,t)$ of the equation \eqref{gH} with data $u_0$ defined on $[0,T]$ for some $T>0$, and such that 
\begin{align*}
u\in C([0,T];\dot{H}^{s_c}(\R^N))\cap L^{q}([0,T]; \dot{W}^{s_c,r}(\R^N)),	
\end{align*}
where the pair $(q,r)$ is the following $L^2$-admissible pair 
\begin{equation}\label{pair}
(q,r) = \left(2p,\frac{2Np}{Np-2}\right).
\end{equation}
Moreover, for all $0<\widetilde{T}<T$ there exists a neighborhood $U$ of $u_0$ in $\dot{H}^{s_c}(\R^N)$ such that the map 
$$
U\rightarrow C([0,T];\dot{H}^{s_c}(\R^N))\cap L^{q}([0,T]; \dot{W}^{s_c,r}(\R^N)),\quad \tilde{u}_0\mapsto\tilde{u}(t),
$$  
is Lipschitz.
\end{theorem}
Note that the above theorem holds regardless of the focusing or defocusing cases; as a consequence, the same result holds in the inhomogeneous space $H^{s_c}$, see Theorem \ref{Lwp}.

We then ask if it is possible to extend the local existence to the larger time intervals, and one of the consequences is the small data theory, which is our next result. 

\begin{theorem}\label{sd-gext}
Let $0<b<N$, $N \geq 1$ and $p\geq 2$ so that $s_c\geq 0$. 
Assume in addition that if $p$ is not an even integer, then $s_c< p-1$. 
Let $u_0\in \dot{H}^{s_c}(\R^N)$ with $\|u_0\|_{\dot{H}^{s_c}} \leq A$. 
Then there exists $\delta = \delta(A)>0$ such that if $\|e^{it\Delta}u_0\|_{S(\dot{H}^{s_c})} \leq \delta$, then there exists a unique global solution $u$ of \eqref{gH} in $\dot{H}^{s_c}(\R^N)$ such that
	\begin{equation}\label{base-level}
	\|u\|_{S(\dot{H}^{s_c})} \leq 2\|e^{it\Delta}u_0\|_{S(\dot{H}^{s_c})},
	\end{equation}
	and
	\begin{equation}\label{E:Hs-level}
	\| |\nabla|^{s_c}u\|_{S(L^2)} \leq 2\,c_1\, \|u_0\|_{\dot{H}^{s_c}}.
	\end{equation}
\end{theorem}

As the small data global existence is available, one may ask if the global existence can be extended for large solutions, or if there is a threshold for global existence. In \cite{AKAR1} we showed a dichotomy for scattering vs. finite time blow-up solutions provided the initial data is in $H^1$; the threshold was given by a combination of the mass-energy and the gradient comparison to that of the ground state (see also Section 5, Theorem \ref{dichotomy}). For the $\dot{H}^s$ data, it is a more difficult question as the conserved quantities at the $\dot{H}^s$ level are not available (unless $s=0$ or $s=1$). Nevertheless, one can still ask for a criteria for finite time blow-up, which we investigate next. Note that if initial data is in $\dot{H}^{s_c} \cap {H}^1$, then it stays in that regularity as the consequence of conservation laws. 
We give a sufficient condition for finite-time blow-up in the generalized Hartree equation \eqref{gH}, which follows the ideas in \cite{HPR10, DR15, L-02, L-10} except that now we find a bound for the convolution term.  To state the result we define the variance, $V(t) \defeq \|xu(t)\|_{L^2(\R^N)}^2$.  

\begin{theorem}\label{blowupL}
Let $u_0 \in H^1$ if $s_c\leq 1$ and $u_0 \in H^{s_c}$ if $s_c>1$. Assume also $V(0) < \infty$ and $E[u] > 0$. The following is a sufficient condition for the blow-up in finite time for the solutions to the gHartree equation \eqref{gH} with initial data $u_0$ in the mass-supercritical case ($s_c > 0$): 
\begin{equation}\label{blowupL1}
\frac{\partial_t\,V(0)}{\omega M[u_0]} < 4\sqrt{2} \, f\left(\frac{E[u_0]V(0)}{(\omega M[u_0])^2}\right), 
\end{equation}
where $~\displaystyle \omega^2 = \frac{N^2(N(p-2)+b-2)}{8(N(p-2)+b)}$  
and the function $f$ is defined as (here, $k = s_c(p-1)$)
	\begin{equation}\label{blowupL2}
	f(x) = \begin{cases}
	\sqrt{\frac{1}{kx^k}+x-\frac{1+k}{k}}\quad\text{if}\,\,\,0 < x < 1\\
	-\sqrt{\frac{1}{kx^k}+x-\frac{1+k}{k}}\quad\text{if}\,\,\,x \geq 1.
	\end{cases}
	\end{equation}
\end{theorem}   

Lastly, we use examples of Gaussian initial data to show known thresholds for global vs. finite existence and  scattering in the following cases: energy-subcritical (see Figure \ref{F:1}), energy-critical (see Figure \ref{F:2}) and energy-supercritical case (Figure \ref{F:3}). 

The paper is organized as follows: in the next section we set the notation and review basic tools, in Section \ref{LWP} we give local well-posedness and the small data theory, in Section \ref{main1} we discuss the sufficient condition for blow-up, and finally, in Section \ref{eg}, we recall dichotomy results and give examples of various thresholds for Gaussian initial data in energy-subcrtical, critical and super-critical cases. 

\textbf{Acknowledgements.} S.R. was partially supported by the NSF CAREER grant DMS-1151618 and also by NSF grant DMS-1815873.  A.K.A.'s graduate research support was in part funded by the grants DMS-1151618 and DMS-1815873 (PI: Roudenko).

\section{Notation and Basic Estimates}

We start with recalling the Fourier transform on $\R^N$ and our convention on normalization $
\hat{f}(\xi) = \frac{1}{(2\pi)^{N/2}}\int_{\R^N}e^{-ix\xi}f(x)\,dx.
$
The homogeneous Sobolev $\dot{H}^{s}$ space is equipped with the norm 
$\|u\|_{\dot{H}^s} = \| |\nabla|{^s} u \|_{L^2(\mathbb{R}^N)}$, where the  operator $|\nabla|^s$ is defined as $\widehat{|\nabla|^{s}f} (\xi) = |\xi|^s \hat{f}(\xi)$. For $f(z)=|z|^{p-2}z$ we note (e.g., refer to \cite{AKAR1}) that
\begin{equation}\label{sdc1}
|f(z_1)-f(z_2)|\lesssim (|z_1|^{p-2}+|z_2|^{p-2})|z_1-z_2| \quad \mbox{for} \quad p \geq 2.
\end{equation}
Also, 
\begin{equation}\label{sdc5}
|f(z_1)\bar{z}_1 - f(z_2)\bar{z}_2| \equiv	| |z_1|^p-|z_2|^p|\lesssim (|z_1|^{p-1}+|z_2|^{p-1})|z_1-z_2| \quad \mbox{for }\quad p \geq 1.
\end{equation}

\begin{definition}
\label{admissible}
The pair $(q,r)$ is called admissible if 
\begin{equation*}
\frac{2}{q} + \frac{N}{r} = \frac{N}{2}\quad \text{and}\quad 2\leq q,r\leq\infty\,\,\,\text{provided}\,\,\,(q,r,N)\neq (2,\infty,2).	
\end{equation*}
\end{definition} 
	
Next, recall the well-known Strichartz estimates (see \cite{KT98}, \cite{C03}, \cite{F05}).
\begin{lemma}\label{strichartz}
Let $s\geq 0$, $(q,r)$ be admissible, $I$ be a compact interval of $\R$ and $u$ be a solution to $iu_t + \Delta u = F(u)$. Then, for any $t_0\in I$
	$$
	\||\nabla|^su\|_{L_t^qL_x^r(I\times\R^N)}\lesssim \||\nabla|^su(t_0)\|_{L^2_x(\R^N)} + \||\nabla|^sF(u)\|_{L^{q^{\prime}}_tL^{r^{\prime}}_x(I\times\R^N)}.
	$$ 
\end{lemma}

For convenience we also define 
$$
\|u\|_{S(L^2)} \defeq \sup_{(q,r)\in \mathcal{A}}\|u\|_{L_t^qL_x^r}
\quad \mbox{and} \quad \|u\|_{S^{\prime}(L^2)} \defeq \inf_{(q,r)\in \mathcal{A}} 
\|u\|_{L_t^{q^{\prime}}L_x^{r^{\prime}}}
$$
with  $\frac{1}{q} + \frac{1}{q^{\prime}} = 1$, $\frac{1}{r} + \frac{1}{r^{\prime}} = 1$, and 
$\mathcal{A}$ to be the set of all admissible pairs in dimension $N=1$ and $N \geq 3$, and in dimension $N=2$ all admissible pairs with $r \leq M$ for some large $M<\infty$ (to have finite $\sup$).
\smallskip

Next, we recall some fractional calculus results used in local well-posedness.
\begin{lemma}[Proposition 3.1 in \cite{CW91}]\label{frchain}
	Suppose $G\in C^1(\C)$ and $s\in(0,1]$. Let $1< q,q_1,q_2<\infty$ are such that $\frac{1}{q}=\frac{1}{q_1}+\frac{1}{q_2}$. Then,
	$$
	\||\nabla|^sG(u)\|_{L_x^q(\R^N)}\lesssim \|G^{\prime}(u)\|_{L_x^{q_1}(\R^N)}\||\nabla|^su\|_{L_x^{q_2}(\R^N)}.
	$$
\end{lemma}
\begin{lemma}[Proposition A.1 in \cite{Visan06}]\label{frac}
Let $G$  be a H\"older continuous function of order $0<\rho<1$. Then, for every $0<s<\rho$, $1<q<\infty$, and $s/\rho < \sigma < 1$, we have
$$
\||\nabla|^sG(u)\|_{L^q_x(\R^N)}\lesssim \||u|^{\rho-\frac{s}{\sigma}}\|_{L^{q_1}_x(\R^N)}\||\nabla|^{\sigma}u\|_{L_x^{\frac{s}{\sigma}q_2}}^{\frac{s}{\sigma}}, 
$$ 
provided $\frac{1}{q}= \frac{1}{q_1} + \frac{1}{q_2}$ and $\big(1-\frac{s}{\rho\sigma}\big)q_1 > 1$.
\end{lemma}  
We also have the following corollary as a consequence of Lemma \ref{frchain} and  Lemma \ref{frac} along with interpolation. 
\begin{corollary}[Corollary 2.7 in \cite{KV10}]\label{noBesov}
	Let $F(u)=|u|^{p-2}u$ with $p\geq 2$ and let $s>1$ if $p$ is an even integer or  $1<s<p-1$ otherwise. Then
	$$
	\||\nabla|^sF(u)\|_{S^{\prime}(L^2)} \lesssim \||\nabla|^su\|_{S(L^2)}\|u\|^{p-2}_{L_{t,x}^{\frac{(p-2)(N-2)}{2}}}.
	$$
\end{corollary}

Next, we recall the Hardy-Littlewood-Sobolev inequality.
\begin{lemma}[Hardy-Littlewood-Sobolev inequality, \cite{Lieb83}]\label{HLS}
	For $0< b <N$ there exists a sharp constant $c_{r_2,b,N}>0$ such that
		$$
		\||x|^{-b}\ast f\|_{L^{r_1}(\R^N)} \leq c_{r_2,b,N}\|f\|_{L^{r_2}(\R^N)},
		$$  
		where $\frac{1}{r_2} + \frac{b}{N} = 1 + \frac{1}{r_1}$ and $1<r_1,\,r_2,\,\frac{N}{b} < \infty$. 
\end{lemma}

\section{Local well-posedness at the critical regularity } 
\label{LWP}
In this section we discuss the local theory for solutions of  the equation \eqref{gH} in $\dot{H}^{s_c}(\R^N)$, for any $s_c\geq 0$. We consider the integral representation of \eqref{gH} with $u_0\in \dot{H}^{s_c}(\R^N)$ and 
$p\geq 2$:
\begin{equation}\label{duhamel}
u(x,t) = e^{it\Delta}u_0 + i \int_0^t e^{i(t-t')\Delta} (|x|^{-b}\ast |u|^p)|u|^{p-2}\, u(t')\,dt'.
\end{equation}
We also require that the nonlinearity power $p$ satisfies an additional constraint, $s_c<p-1$ if $p$ is not an even integer. This ensures that one can take the derivative of $|u|^{p-2}u$ term $s_c$ times. We obtain local well-posedness via contraction argument strategy using Strichartz estimates. 

\begin{proof} \textit{(of Theorem \ref{lwp})} For $T>0$ and $M>0$ determined later, let  
	\begin{align*}
	B_T=\{u\in  C([0,T]);\dot{H}^{s_c}(\R^N)\cap L^{q}([0,T]);\dot{W}^{s_c,r}(\R^N)\,:&\,\,\||\nabla|^{s_c}u\|_{L_t^qL_x^r([0,T]\times\R^N)}\leq M \}.
	\end{align*}
	We prove that the following operator 
	\begin{align}\label{Dope1}
	\Phi(u(t))=e^{it\Delta}u_0 + i \int_0^t e^{i(t-t')\Delta}N(u(t')) \,dt'
	\end{align} 
	is a contraction on the set $B_T$ for some $T>0$. Here, \begin{equation}\label{NL}
		N(u(\cdot))=(|x|^{-b}\ast |u|^p)|u|^{p-2}u(\cdot).
	\end{equation} 
	Denoting by $I=[0,T]$ and using Strichartz estimates, we obtain
	\begin{align}\label{lwp2}
	\||\nabla|^{s_c}\Phi(u(t))\|_{L_I^{q}L_x^{r}}&\leq  c_1\|u_0\|_{\dot{H}^{s_c}}+c_1\left(\int_{0}^{T}\||\nabla|^{s_c} N(u)\|^{q^{\prime}}_{L_x^{r^{\prime}}}\,dt\right)^{1/q^{\prime}}.
	\end{align} 
	Using H\"older's inequality, Lemma \ref{HLS} (for $b+s_c<N$ for $A_1$) and Sobolev inequality (take $r_s = \frac{2Np(p-1)}{(N-b)p+2}$ and observe that $r_s>r$, in particular, $r<N/s_c$), we get
	\begin{align}\label{lwp4}
	\||\nabla|^{s_c} N(u)\|_{L_I^{q^{\prime}}L_x^{r^{\prime}}} \leq A_1 + A_2,
	\end{align}
	 where $A_1$ is estimated as follows  
	\begin{align}\label{lwp4a}\notag
	A_1 &\leq \|(|x|^{-(b+s_c)}\ast|u|^p)\|_{L_I^{2}L_x^{\frac{2N}{b}}}\| u\|^{p-1}_{L_I^{q}L_x^{r_s}}\\
	&\leq c_{p,b,N}\|u\|^p_{L_I^{q}L_x^{r_s}}\||\nabla|^{s_c}u\|^{p-1}_{L_I^{q}L_x^{r}}\leq c_{p,b,N}\||\nabla|^{s_c}u\|^{2p-1}_{L_I^{q}L_x^{r}},
	\end{align}
	and for $A_2$, we use Corollary \ref{noBesov} along with H\"older's inequality, Lemma \ref{HLS} and Sobolev inequality to obtain
	\begin{align}\label{lwp4b}\notag
	A_2 &\leq \|(|x|^{-b}\ast|u|^p)\|_{L_I^{2}L_x^{\frac{2N(p-1)}{2+b(p-2)-N(p-2) }}}\|u\|^{p-2}_{L_I^{q}L_x^{r_s}}\||\nabla|^{s_c}u\|_{L_I^{q}L_x^{r}}\\
	&\leq c_{p,b,N}\|u\|^p_{L_I^{q}L_x^{r_s}}\||\nabla|^{s_c}u\|^{p-1}_{L_I^{q}L_x^{r}}\leq c_{p,b,N}\||\nabla|^{s_c}u\|^{2p-1}_{L_I^{q}L_x^{r}}.
	\end{align}
	Thus, \eqref{lwp4} yields
	\begin{equation}\label{lwp4c}
	\||\nabla|^{s_c}N(u)\|_{L_I^{q^{\prime}}L_x^{r^{\prime}}}\leq 2c_{p,b,N}\||\nabla|^{s_c}u\|^{2p-1}_{L_I^{q}L_x^{r}}.
	\end{equation}		
	Substituting \eqref{lwp4c} in \eqref{lwp2}, we get
	\begin{equation}\label{lwp2a}
	\||\nabla|^{s_c}\Phi(u(t))\|_{L_I^{q}L_x^{r}}\leq c_1\|u_0\|_{\dot{H}^{s_c}} + 2c_1c_{p,b,N}\||\nabla|^{s_c}u\|^{2p-1}_{L_I^{q}L_x^{r}}.
	\end{equation}
	Following a similar argument, we also obtain 
	\begin{align}\label{lwp2b}
	\|\Phi(u(t))\|_{L_I^{\infty}\dot{H}_x^{s_c}}\leq c_1\|u_0\|_{\dot{H}^{s_c}_x}+2c_1c_{p,b,N} \||\nabla|^{s_c}u\|^{2p-1}_{L_I^{q}L_x^{r}}.
	\end{align}
	Take $\|u_0\|_{\dot{H}^{s_c}_x}$ small enough so that
	\begin{align}\label{lwp6}
	\left(2\|u_0\|_{\dot{H}^{s_c}_x}\right)^{2(p-1)}\leq \frac{1}{4^{2p}(c_1)^{2p-1}c_{p,b,N}}.
	\end{align}
	Set $M=8c_1\|u_0\|_{\dot{H}^{s_c}_x}$. Thus, 
	$$
	M^{2(p-1)} \leq \frac{1}{16\,c_1\,c_{p,b,N}}.
	$$
	 Take $T>0$ such that $\|e^{it\Delta}(|\nabla|^{s_c}u_0)\|_{S(L^2;[0,T])}\leq \frac{M}{4}$.
	Thus, from estimates \eqref{lwp2a} and \eqref{lwp2b} we have that for $u\in B_T$ with $T$ as above
	\begin{align}\label{lwp5}\notag
	\|\Phi(u(t))\|_{L_I^{\infty}\dot{H}_x^{s_c}} + \|\Phi(u(t))\|_{L_I^{q_1}\dot{W}^{s_c,r}_{x}}&\leq 2c_1\|u_0\|_{\dot{H}^{s_c}_x}+4\,c_1\,c_{p,b,N}M^{2p-1}\\
	&\leq \frac{M}{4} + 4\,c_1\,c_{p,b,N}\,\frac{M}{16\,c_1\,c_{p,b,N}}=\frac{M}{4} + \frac{M}{4} < M,
	\end{align}
yielding $\Phi$ mapping $B_T$ into itself. 
	
	To complete the proof we need to show that the operator $\Phi$ is a contraction. This can be achieved by running the same argument as above on the difference 
	$$
	d(\Phi(u(t)),\Phi(v(t))):=\||\nabla|^{s_c}[\Phi(u(t))-\Phi(v(t))]\|_{L_I^{q}L^{r}_x},
	$$
	for $u,v\in B_T$. We first apply Strichartz estimates to get
	$$
	d(\Phi(u(t)),\Phi(v(t)))\leq c_1 \||\nabla|^{s_c}[\Phi(u(t))-\Phi(v(t))]\|_{L_I^{q^{\prime}}L_x^{r^{\prime}}},
	$$
	where
	\begin{align*}
	\||\nabla|^{s_c}[\Phi(u(t))-\Phi(v(t))]\|_{L_I^{q^{\prime}}L_x^{r^{\prime}}}\leq&\,\,\||\nabla|^s\big[\big(|x|^{-b}\ast |u|^p\big)\big(|u|^{p-2}u-|v|^{p-2}v\big)\big]\|_{L_I^{q^{\prime}}L_x^{r^{\prime}}}\\
	&+\||\nabla|^s\big[\big(|x|^{-b}\ast (|u|^p-|v|^p)\big)|v|^{p-2}v\big]\|_{L_I^{q^{\prime}}L_x^{r^{\prime}}}\\
	=&\,\,D_1+D_2.
	\end{align*}
	For $D_1$ we first use the fractional product rule
	\begin{align*}
	D_1\leq&\,\,\|(|x|^{-(b+s_c)}\ast|u|^p)\|_{L_I^{2}L_x^{\frac{2N}{b}}}\||u|^{p-2}u-|v|^{p-2}v\|_{L_I^{\frac{q}{p-1}}L_x^{\frac{r_s}{p-1}}}\\
	&+\|(|x|^{-b}\ast|u|^p)\|_{L_I^{2}L_x^{\frac{2N(p-1)}{2+b(p-2)-N(p-2)}}}\||\nabla|^{s_c}(|u|^{p-2}u-|v|^{p-2}v)\|_{L_I^{\frac{q}{p-1}}L_x^{\frac{2Np(p-1)}{Np(2p-3)-bp(p-2)-2}}},
	\end{align*}
	then using the similar calculations as in \eqref{lwp4a} and \eqref{lwp4b} together with \eqref{sdc1} yields
	\begin{align}\label{lwp13}
	D_1\leq 2c_{p,b,N}\||\nabla|^{s_c}u\|^{p}_{L_I^{q}L_x^{r}}\Big(\||\nabla|^{s_c}u\|^{p-2}_{L_I^{q}L_x^{r}}+\||\nabla|^{s_c}v\|^{p-2}_{L_I^{q}L_x^{r}}\Big)\||\nabla|^{s_c}(u-v)\|_{L_I^{q}L_x^{r}}.
	\end{align}
	Again using the fractional product rule, we have
	\begin{align*}
	D_2\leq&\,\,\||x|^{-(b+s_c)}\ast(|u|^p-|v|^p)\|_{L_I^{2}L_x^{\frac{2N}{b}}}\|v\|^{p-1}_{L_I^{q}L_x^{r_s}}\\
	&+\||x|^{-b}\ast(|u|^p-|v|^p)\|_{L_I^{2}L_x^{\frac{2N(p-1)}{ 2+b(p-2)-N(p-2)}}}\|v\|^{p-2}_{L_I^qL_x^{r_s}}\||\nabla|^{s_c}v\|_{L_I^{q}L_x^{r}}.
	\end{align*}
	Using the similar calculations as above along with \eqref{sdc5}, we get
	\begin{align}\label{lwp14}
	D_2\leq 2c_{p,b,N}\Big(\||\nabla|^{s_c}u\|^{p-1}_{L_I^{q}L_x^{r}}+\||\nabla|^{s_c}v\|^{p-1}_{L_I^{q}L_x^{r}}\Big)\||\nabla|^{s_c}(u-v)\|_{L_I^{q}L_x^{r}}\||\nabla|^{s_c}v\|^{p-1}_{L_I^{q}L_x^{r}}.
	\end{align}
	Combining \eqref{lwp13} and \eqref{lwp14}, we obtain that for $u,v\in B_T$
	\begin{align*}
	d(\Phi(u(t)),\Phi(v(t)))\leq 8c_1c_{p,b,N}M^{2(p-1)}d(u,v).
	\end{align*}
	Taking $M$ and $T$ as in \eqref{lwp5} together with \eqref{lwp6} implies that $\Phi$ is a contraction on $B_T$. Now continuous dependence with respect to $u_0$ is a direct consequence of the above estimates, we note that if $u$ and $v$ are the corresponding solutions of \eqref{duhamel} with initial data $u_0$ and $v_0$, respectively, then
	$$
	u(t)-v(t) = e^{it\Delta}(u_0-v_0) + i\int_{0}^{t}e^{i(t-t')\Delta}(N(u)-N(v))(t') \,dt'.
	$$
	Thus, the same argument as in \eqref{lwp13} and \eqref{lwp14} yields 
	\begin{align*}
	d(u(t),v(t))\leq c_1\|u_0-v_0\|_{\dot{H}^{s_c}_x}+\frac{1}{2}d(u(t),v(t)).
	\end{align*}
	This implies that if $\|u_0-v_0\|_{H^s_x}$ is small enough (see \eqref{lwp6}), we have that
	\begin{align*}
	d(u(t),v(t))\leq 2c_1\|u_0-v_0\|_{\dot{H}^{s_c}_x},
	\end{align*} 
	and this completes the proof.
\end{proof}
We now show the inhomogeneous version of Theorem \ref{lwp}.
\begin{theorem}\label{Lwp}
	Let $0<b<N$, $N \geq 1$ and $p\geq 2$ so that $s_c\geq 0$. 
	Assume in addition that if $p$ is not an even integer, then $s_c< p-1$. 
	Let $u_0\in H^{s_c}(\R^N)$. 
	Then there exists a unique solution $u(x,t)$ of the equation \eqref{gH} with data $u_0$ defined on $[0,T]$ for some $T>0$, and such that 
	\begin{align}\label{Lwpspace}
	u\in \mathcal{B}:= C([0,T];H^{s_c}(\R^N))\cap L^{q}([0,T]; W^{s_c,r}(\R^N)),	
	\end{align}
	where the pair $(q,r)$ is the $L^2$-admissible pair given by \eqref{pair}. Moreover, for all $0<\widetilde{T}<T$ there exists a neighborhood $U$ of $u_0$ in $H^{s_c}(\R^N)$ such that the map $U\rightarrow \mathcal{B}$, $\tilde{u}_0\mapsto\tilde{u}(t),$ is Lipschitz.
\end{theorem}

\begin{proof} For $T>0$ and $M>0$ determined later, let  
	\begin{align*}
	\mathcal{B}_T = \{u\in  \mathcal{B}\,:&\,\,\||\nabla|^{s_c}u\|_{L_t^qL_x^r([0,T]\times\R^N)}\leq 2M \,\,\,\text{and}\,\,\|u\|_{L_t^qL_x^r([0,T]\times\R^N)}\leq 2c_1\|u_0\|_{L^2} \}.
	\end{align*}
	We prove that $\Phi$ defined in \eqref{Dope1} is a contraction on the set $\mathcal{B}_T$ for some $T>0$. Denoting again $I=[0,T]$, using \eqref{lwp2} for $\||\nabla|^{s_c}u\|_{L_I^qL_x^r}$ and estimating the inhomogeneous part using Strichartz estimates, we have
	\begin{align}\label{Lwp1}
	\|\Phi(u(t))\|_{L^{q}_IL_x^{r}}\leq c_1\| u_0\|_{L_x^2}+c_1\left(\int_{0}^{T}\| N(u)\|^{q^{\prime}}_{L_x^{r^{\prime}}}\,dt\right)^{1/q^{\prime}}.
	\end{align}
	Using H\"older's inequality, Lemma \ref{HLS} and Sobolev inequality, we estimate
	\begin{align}\label{Lwp3}\notag
	\|N(u)\|_{L_I^{q^{\prime}}L_x^{r^{\prime}}} &\leq \|(|x|^{-b}\ast|u|^p)\|_{L_I^{2}L_x^{\frac{2N}{b}}}\| u\|^{p-1}_{L_I^{q}L_x^{r_s}}\leq c_{p,b,N} \|u\|^p_{L_I^{2p}L_x^{\frac{2Np}{2N-b}}}\||\nabla|^{s_c}u\|^{p-1}_{L_I^{q}L_x^{r}}\\
	&\leq c_{p,b,N}\|u\|^{p-1}_{L_I^{q}L_x^{r_s}}\|u\|_{L_I^{q}L_x^{r}}\||\nabla|^{s_c}u\|^{p-1}_{L_I^{q}L_x^{r}}\leq c_{p,b,N} \||\nabla|^{s_c}u\|^{2(p-1)}_{L_I^{q}L_x^{r}} \|u\|_{L_I^{q}L_x^{r}}.
	\end{align}
	Using \eqref{Lwp3}, we write \eqref{Lwp1} as
	\begin{equation}\label{Lwp1a}
	\|\Phi(u(t))\|_{L^{q}_IL_x^{r}}\leq c_1\| u_0\|_{L_x^2} + c_1c_{p,b,N}\||\nabla|^{s_c}u\|^{2(p-1)}_{L_I^{q}L_x^{r}} \|u\|_{L_I^{q}L_x^{r}}.
	\end{equation}
Then for $u\in\mathcal{B}_T$, we have
\begin{equation}\label{Lwp1b}
	\|\Phi(u(t))\|_{L^{q}_IL_x^{r}}\leq c_1\| u_0\|_{L_x^2}\left(1 + 2^{2p-1}c_1c_{p,b,N}M^{2(p-1)}\right).
\end{equation}
Next, invoking \eqref{lwp2a} for $u\in\mathcal{B}_T$, we have
\begin{equation}\label{Lwp2}
\||\nabla|^{s_c}\Phi(u(t))\|_{L^{q}_IL_x^{r}}\leq c_1\| u_0\|_{\dot{H}^{s_c}_x} + 2c_1c_{p,b,N}(2M)^{2p-1}.
\end{equation}
Take $\|u_0\|_{\dot{H}^{s_c}_x}$ small enough so that
\begin{align}\label{Lwp6}
\|u_0\|^{2(p-1)}_{\dot{H}^{s_c}_x}\leq \frac{1}{2^{2(p+1)}(c_1)^{2p-1}c_{p,b,N}}.
\end{align}
Set $M=c_1\|u_0\|_{\dot{H}^{s_c}_x}$. Thus,
\begin{equation}\label{Lwp7}
	M^{2(p-1)} \leq \frac{1}{2^{2(p+1)}\,c_1\,c_{p,b,N}}. 
\end{equation}
Take $T>0$ such that $\|e^{it\Delta}(|\nabla|^{s_c}u_0)\|_{S(L^2;[0,T])}\leq \frac{M}{2}$. Thus, using \eqref{Lwp6} and \eqref{Lwp7} on \eqref{Lwp1b} and \eqref{Lwp2}, we have 
\begin{equation}\label{Lwp8}
	\|\Phi(u(t))\|_{L^{q}_IL_x^{r}}\leq c_1\|u_0\|_{L^2_x}(1+\frac{1}{8})<2c_1\|u_0\|_{L^2_x}\quad\text{and}\quad\||\nabla|^{s_c}\Phi(u(t))\|_{L^{q}_IL_x^{r}}\leq M+\frac{M}{4}<2M. 
\end{equation}
A similar argument as in \eqref{Lwp8} yields 
	\begin{align}\label{Lwp9}
	\|\Phi(u(t))\|_{L_I^{\infty}L_x^{2}}\leq 2c_1\|u_0\|_{L^2_x}\quad\text{and}\quad\||\nabla|^{s_c}\Phi(u(t))\|_{L^{\infty}_IL_x^{2}}\leq 2M .
	\end{align}
Hence, \eqref{Lwp8} and \eqref{Lwp9} implies that $\Phi$ maps $\mathcal{B}_T$ into itself. The rest of the argument as in the proof of Theorem \ref{lwp}.	
\end{proof}

Since we now have $\dot{H}^{s_c}$ local well-posedness, we investigate the global existence of small data in $\dot{H}^{s_c}$ for $s_c \geq 0$. 
\begin{proof} \textit{(of Theorem \ref{sd-gext}.)}
	Denote
	\begin{align*}
		B = \Big\{ u\, :\, \|u\|_{S(\dot{H}^{s_c})}& \leq 2\,\|e^{it\Delta}u_0\|_{S(\dot{H}^{s_c})} \quad\text{and}\quad \||\nabla|^{s_c}u\|_{S(L^2)} \leq 2\,c\|u_0\|_{\dot{H}^{s_c}}\Big\},
	\end{align*}
	and define
	\begin{align}\label{sd0}
	\Phi_{u_0}(u) = e^{it\Delta}u_0 + i \int_0^t e^{i(t-t')\Delta} N(u(t'))\,dt',
	\end{align}
	where $N(u(\cdot))$ as in \eqref{NL}. Applying the triangle inequality and Strichartz estimates to \eqref{sd0}, we obtain
	\begin{align}\label{sd1}
	\|\Phi_{u_0}(u)\|_{S(\dot{H}^{s_c})} \leq \|e^{it\Delta}u_0\|_{S(\dot{H}^{s_c})} + c_1\||\nabla|^{s_c}N(u)\|_{L_t^{q^{\prime}}L_x^{r^{\prime}}},
	\end{align}
	and 
	\begin{equation}\label{sd2}
		\||\nabla|^{s_c}\Phi_{u_0}(u)\|_{S(L^2)} \leq c_1\|u_0\|_{\dot{H}^{s_c}} + c_1\||\nabla|^{s_c}N(u)\|_{L_t^{q^{\prime}}L_x^{r^{\prime}}}.
	\end{equation}
	 Using the estimates from Theorem \ref{lwp}, we get
	$$
	\||\nabla|^{s_c}N(u)\|_{L_t^{q^{\prime}}L_x^{r^{\prime}}}\leq 2\,c_{p,b,N}\|u\|^{2(p-1)}_{S(\dot{H}^{s_c})}\||\nabla|^{s_c}u\|_{S(L^2)}.
	$$
	Therefore, \eqref{sd1} gives
	\begin{align}\label{sd3}
		\|\Phi_{u_0}(u)\|_{S(\dot{H}^{s_c})} \leq \|e^{it\Delta}u_0\|_{S(\dot{H}^{s_c})} + 2c_1c_{p,b,N}\|u\|^{2(p-1)}_{S(\dot{H}^{s_c})}\||\nabla|^{s_c}u\|_{S(L^2)},
	\end{align}
	and \eqref{sd2} gives
	\begin{align}\label{sd4}
		\||\nabla|^{s_c}\Phi_{u_0}(u)\|_{S(L^2)} \leq c_1\|u_0\|_{\dot{H}^{s_c}} + 2c_1c_{p,b,N}\|u\|^{2(p-1)}_{S(\dot{H}^{s_c})}\||\nabla|^{s_c}u\|_{S(L^2)}.
	\end{align}
	Thus, from \eqref{sd3} for $u\in B$, we obtain 
	\begin{align*}
	\|\Phi_{u_0}(u)\|_{S(\dot{H}^{s_c})} \leq \|e^{it\Delta}u_0\|_{S(\dot{H}^{s_c})}\left(1+c_1\,c_{p,b,N}\, 2^{2p}\, \|e^{it\Delta}u_0\|_{S(\dot{H}^{s_c})}^{2p-3}\,A\right),
	\end{align*}
	which implies that we need
	\begin{equation}\label{sd5}
		c_1\,c_{p,b,N} \, 2^{2p}\, \|e^{it\Delta}u_0\|_{S(\dot{H}^{s_c})}^{2p-3}\,A\leq 1.
	\end{equation}
	And, from \eqref{sd4} for $u\in B$, we obtain
	\begin{align*}
		\||\nabla|^{s_c}\Phi_{u_0}(u)\|_{S(L^2)} \leq c_1\|u_0\|_{\dot{H}^{s_c}} \left(1+2^{2p}\,c_{p,b,N} \|e^{it\Delta}u_0\|_{S(\dot{H}^{s_c})}^{2(p-1)}\right),
	\end{align*}
	which implies that we require
	\begin{equation}\label{sd6}
	c_{p,b,N} \, 2^{2p}\, \|e^{it\Delta}u_0\|_{S(\dot{H}^{s_c})}^{2(p-1)}\leq 1.
	\end{equation}
	Therefore, from \eqref{sd5} and \eqref{sd6}, choosing 
	$$
	\delta < \delta_0 \leq \min\left(\frac{1}{\sqrt[2p-3]{2^{2p+1}c_1c_{p,b,N}A}},\frac{1}{\sqrt[2(p-1)]{2^{2p+1}c_{p,b,N}}}\right),
	$$ 
	implies that $\Phi_{u_0}\in B$. Now we show that $\Phi_{u_0}(u)$ is a contraction on $B$ with the metric
	$$
	d(u,v)=\||\nabla|^{s_c}(u-v)\|_{S(L^2)}.
	$$ 
	For $u$, $v\in B$, by Strichartz estimates, we obtain
	\begin{align}\label{sd13}
	\||\nabla|^{s_c}[\Phi_{u_0}(u)-\Phi_{u_0}(v)]\|_{S(L^2)}&\leq c_1\||\nabla|^{s_c}[N(u)-N(v)]\|_{L_t^{q^{\prime}}L_x^{r^{\prime}}}.
	\end{align}
	The triangle inequality applied to the term on the right-hand side of 
	\eqref{sd13} yields 
	\begin{align*}
	\||\nabla|^{s_c}[N(u)-N(v)]\|_{L_t^{q^{\prime}}L_x^{r^{\prime}}}\leq&\,\, \||\nabla|^{s_c}[\big(|x|^{-b}\ast |u|^p\big)\big(|u|^{p-2}u-|v|^{p-2}v\big)]\|_{L_t^{q^{\prime}}L_x^{r^{\prime}}}\\
	&+\||\nabla|^{s_c}[\big(|x|^{-b}\ast (|u|^p-|v|^p)\big)|v|^{p-2}v]\|_{L_t^{q^{\prime}}L_x^{r^{\prime}}}.
	\end{align*} 
	Using the estimates \eqref{lwp13} and \eqref{lwp14}, we obtain
	\begin{align*}
	\||\nabla|^{s_c}[N(u)-N(v)]\|_{L_t^{q^{\prime}}L_x^{r^{\prime}}}	\leq\,\,2&\,c_{p,b,N}\Big[\|u\|^p_{S(\dot{H}^{s_c})}\Big(\|u\|^{p-2}_{S(\dot{H}^{s_c})}+\|v\|^{p-2}_{S(\dot{H}^{s_c})}\Big)\\
	&+\Big(\|u\|^{p-1}_{S(\dot{H}^{s_c})}+\|v\|^{p-1}_{S(\dot{H}^{s_c})}\Big)\|v\|^{p-1}_{S(\dot{H}^{s_c})}\Big]\||\nabla|^{s_c}(u-v)\|_{S(L^2)}.
	\end{align*}
	For $u$, $v\in B$, we have that
	\begin{align}\label{sd14}
	\||\nabla|^{s_c}[N(u)-N(v)]\|_{L_t^{q^{\prime}}L_x^{r^{\prime}}}\leq 2^{2p+1}\,c_{p,b,N}\|e^{it\Delta}u_0\|^{2(p-1)}_{S(\dot{H}^{s_c})}\||\nabla|^{s_c}(u-v)\|_{S(L^2)}.
	\end{align}
	Combining \eqref{sd14} with 
	\eqref{sd13}, we get
	$$
	d(\Phi_{u_0}(u),\Phi_{u_0}(v))\leq 2^{2(p+1)}c_1\,c_{p,b,N}\|e^{it\Delta}u_0\|^{2(p-1)}_{S(\dot{H}^{s_c})}d(u,v)\leq \frac{1}{2}d(u,v)
	$$
	for $\delta_1\leq \sqrt[2(p-1)]{\frac{1}{2^{2p+3}c_1c_{p,b,N}}}.$ 
	Finally, taking $\delta \leq \min(\delta_0,\delta_1)$ concludes that $\Phi_{u_0}$ is a contraction.
\end{proof}

\section{Blow-up criterion}\label{main1}

Before proving Theorem \ref{blowupL}, we recall that solutions $u(t)$ of \eqref{gH} with finite variance, $V(0) = \|xu_0\|_{L^2}^2 < \infty$, satisfy the following virial identities
\begin{equation*}
V_t(t) = 4\Im \int_{\R^N} \bar{u}\,x\cdot\nabla u\,dx,
\end{equation*}
\begin{align*}\label{virial2}
V_{tt}(t) = 16 E[u] - \frac{8s_c(p-1)}{p} Z(u) 
\equiv 16(s_c(p-1)+1)E[u] - 8s_c(p-1)\|\nabla u\|_{L^2}^2.
\end{align*}
\begin{proof} \textit{(of Theorem \ref{blowupL})}
Recalling the decomposition (4.1) from \cite{DR15} 
\begin{equation*}
	\frac{N^2}{4} \|u\|_{L^2}^4 + \left|\frac{V_t(t)}{4}\right|^2 \leq V(t)\|\nabla u\|_{L^2}^2,
\end{equation*} 
we obtain 
\begin{equation}\label{L3}
	V_{tt}(t) \leq 16 (s_c(p-1)+1) E[u] - 2s_c(p-1) \frac{N^2(M[u])^2}{V(t)} - \frac{s_c(p-1)}{2}\frac{|V_t(t)|^2}{V(t)}. 
\end{equation}
We rewrite the above by making a substitution $V(t)=B^{\frac{1}{\alpha+1}}(t)$, where $\alpha=\frac{s_c(p-1)}{2}=\frac{N(p-2)+b-2}{4} > 0$, to remove the last term and re-write \eqref{L3} as
\begin{equation*}
	B_{tt}\leq 16(\alpha+1)(2\alpha+1)E[u]B^{\frac{\alpha}{\alpha+1}}-4\alpha(\alpha+1)N^2(M[u])^2B^{\frac{\alpha-1}{\alpha+1}}.
\end{equation*}
Set $\displaystyle \gamma = 4\sqrt{2}\,\frac{E[u]}{\omega\,M[u]}$, where $\displaystyle \omega^2 = \frac{N^2\,\alpha}{4\,(2\alpha+1)} = \frac{N^2(N(p-2)+b-2)}{8(N(p-2)+b)}$  and  introduce the rescaled variables $v$ and the time $s$ as follows: $s=\gamma \, t$ and
$$
B(t) = B_{\max} \, v(s), \quad 
\mbox{where} \quad 
B_{\max}=\left(\frac{(\omega\,M[u])^2}{E[u]}\right)^{\alpha+1}.
$$ 
Then, with these new variables, we obtain
\begin{equation}\label{L6}
	\frac{2}{(\alpha+1)(2\alpha+1)}v_{ss} \leq v^{\frac{\alpha}{\alpha+1}}-v^{\frac{\alpha-1}{\alpha+1}}, \quad s\in[0, T^{*}/\gamma),
\end{equation}
and the equation \eqref{L6} can be written as 
\begin{equation*}
	 v_{ss} \leq -c\,\frac{\partial\widetilde{U}}{\partial v},
\end{equation*}
where $c=\frac{(\alpha+1)(2\alpha+1)}{2}$ and the potential 
$$
\widetilde{U}(v)=\frac{\alpha+1}{2\alpha}v^{\frac{2\alpha}{\alpha+1}} -\frac{\alpha+1}{2\alpha+1}v^{\frac{2\alpha+1}{\alpha+1}}.
$$ 
In fact for some function $g^2(s)>0$, we have
$$
v_{ss} = -c\,\frac{\partial\widetilde{U}}{\partial v} - g^2(s),
$$
and using the same analogy from mechanics as in \cite{L-10}, \cite{HPR10}, \cite{DR15}, let $v(t)$ be the coordinate of the particle with the unit mass moving under two forces: $F_1=-c\,\frac{\partial\widetilde{U}}{\partial v}$ and an unknown external force $F_2=-g^2(t)<0$ (because of the sign, it pulls the particle towards the origin). The collapse occurs if the particle reaches the origin in finite time, i.e., when $v(t^*)=0$ for some $0<t^*<\infty$. If it reaches the origin without the force $F_2 = -g^2(t)$, then it would also reach the origin when this force is applied, thus, leading to the following equation
\begin{equation}\label{L9a}
 \frac{1}{c}\,v_{ss} + \frac{\partial\widetilde{U}}{\partial v}=0.
\end{equation}
The energy of this particle, defined as
\begin{equation}\label{L10}
	\mathcal{E}(s)=\frac{1}{2\,c}v_s^2+\widetilde{U}(v(s)),
\end{equation}
is conserved. Note that the curve for $\widetilde{U}$ is increasing from the origin (for positive $v$) and then decreasing with the local maximum $\widetilde{U}_{\max}=\frac{\alpha+1}{2\alpha(2\alpha+1)}$ attained at $v=1$. Using the energy from \eqref{L10}, we obtain the blow-up conditions for \eqref{L9a} similar to Proposition 4.1 in \cite{DR15}, see also \cite{HPR10}: 
\begin{enumerate}
	\item[(I)]  $\mathcal{E}(0) < \widetilde{U}_{\max}$ and $v(0) <1$, 
	\item[(II)] $\mathcal{E}(0) > \widetilde{U}_{\max}$ and $v_s(0) < 0$,
	\item[(III)] $\mathcal{E}(0) = \widetilde{U}_{\max}$, $v_s(0) < 0$ and $v(0) <1$.
\end{enumerate}
 Define $v=\widetilde{V}^{\alpha+1}$ and rewrite the energy as  
\begin{align*}
\mathcal{E}&=\frac{\alpha+1}{2\alpha+1}\,\widetilde{V}^{2\alpha} \left(\widetilde{V}_s^2 - \widetilde{V} + \frac{2\alpha+1}{2\alpha} \right).
\end{align*}
Observe that 
\begin{equation}\label{L13}
	\mathcal{E} < \widetilde{U}_{\max} \iff \widetilde{V}_s^2 < \frac{1}{2\alpha\,\widetilde{V}^{2\alpha}}+\widetilde{V}-\frac{2\alpha+1}{2\alpha}.
\end{equation}
Let $k=2\alpha=s_c(p-1)$ and set the function
\begin{equation}\label{L14}
	f(x)=\sqrt{\frac{1}{kx^k}+x-\frac{1+k}{k}},
\end{equation}
then the blow-up conditions (I)-(III) with \eqref{L13} and \eqref{L14} are given as 
\begin{equation*}
	\widetilde{V}_s(0) < \begin{cases}
	+ f(\widetilde{V}(0)) & \quad \text{if} \,\,\,\widetilde{V}(0) < 1\\
	- f(\widetilde{V}(0)) & \quad \text{if}\,\,\, \widetilde{V}(0) \geq 1	
	\end{cases} .
\end{equation*}  
Substituting for $v$, $B_{\max}$ in $V(t)=\left(B_{\max}v\right)^{\frac{1}{\alpha+1}}$ yields 
$$
V(t)=\frac{(\omega M[u])^2}{E[u]}\widetilde{V}\left(4\sqrt{2}\,\frac{E[u]}{\omega M[u]}\,t\right),
$$  
and therefore, we obtain
\begin{align}\label{L16}
\frac{V_t(0)}{\omega\,M[u]} <  4\sqrt{2}\, f\left(\frac{E[u]V(0)}{(\omega\,M[u])^2}\right),
\end{align}
as claimed.
\end{proof}

{\it Remark.} For the real-valued initial data, the expression \eqref{L16} can be simplified to
\begin{align}\label{L16real}
V(0) < \frac{(\omega\,M[u])^2}{E[u]}.
\end{align}
Thus, knowing how big the initial variance for the real-valued data is, gives us the way to show that the solution from this initial data will blow-up in finite time. We use this in examples in the next section.

\section{Examples}\label{eg}

In this section we show examples of known thresholds in the energy-subcritical, critical and supercritical cases for the Gaussian initial data. 

\subsection{Review of known thresholds}
Before discussing examples we mention the dichotomy results from \cite{AKAR1} as they are helpful in identifying thresholds for finite vs infinite time existence in the energy-subcritical and critical cases. We recall that for $0<s_c<1$ the quantities $M[u_0]^{1-s_c}E[u_0]^{s_c}$ and $\|u_0\|_{L^2(\R^N)}^{1-s_c} \,  \|\nabla u_0 \|_{L^2(\R^N)}^{s_c}$ are scale-invariant, as it was first observed in \cite{HR07} in the NLS context.  For the following statement, we renormalize them using the sharp constant of the convolution-type Gagliardo-Nirenberg inequality \eqref{GNC}, or equivalently, the $L^2$-norm of ground states to the equation $-\Delta Q + Q - (|x|^{-b} \ast |Q|^{p}) |Q|^{p-2} Q=0$, see discussion on this in Section 4 in \cite{AKAR1} as well as the derivation of the sharp constant. For now note that the sharp constant of the following inequality 
\begin{equation}\label{GNC}
Z(u)\leq C_{GN}\|\nabla u\|_{L^2}^{2s_c(p-1)+2}\|u\|_{L^2}^{2(1-s_c)(p-1)}
\end{equation}
is attained at ground states $Q$ and is equal to $C_{GN}=\|Q\|_{L^2(\mathbb{R}^N)}^{-2(p-1)}$ (note that the value $\|Q\|_{L^2(\R^N)}$ is uniquely determined). In a spirit of NLS and to state Theorem \ref{dichotomy} concisely below, we define (for $s_c>0$)
\begin{equation}\label{norMEG}
	\mathcal{ME}[u] = \frac{M[u]^{\frac{1-s_c}{s_c}}E[u]}{M[Q]^{\frac{1-s_c}{s_c}}E[Q]}\quad\text{and}\quad \mathcal{G}[u(t)]=\frac{\|u\|^{\frac{1-s_c}{s_c}}_{L^2(\R^N)}\|\nabla u(t)\|_{L^2(\R^N)}}{\|Q\|^{\frac{1-s_c}{s_c}}_{L^2(\R^N)}\|\nabla Q\|_{L^2(\R^N)}},
\end{equation}
where the denominators are
$$
\|Q\|^{\frac{1-s_c}{s_c}}_{L^2(\R^N)}\|\nabla Q\|_{L^2(\R^N)} = \left(\frac{p\left(C_{GN}\right)^{-1}}{s_c(p-1)+1}\right)^\frac{1}{2s_c(p-1)}
$$
and 
$$
M[Q]^{\frac{1-s_c}{s_c}}E[Q] = \frac{s_c(p-1)}{2s_c(p-1)+2}\,\left(\|Q\|^{\frac{1-s_c}{s_c}}_{L^2(\R^N)}\|\nabla Q\|_{L^2(\R^N)}\right)^{2}.
$$
In \cite{AKAR1} we proved that in the inter-critical regime $(0 < s_c < 1)$ there is a dichotomy for the $H^1$ solutions under the mass-energy threshold $\mathcal{ME}[u_0]<1$ via the well-known concentration compactness and rigidity method of Kenig-Merle \cite{KM06} following the strategy of \cite{HR08}, \cite{DHR08}, \cite{HR10}, see Theorem \ref{dichotomy}. In \cite{AKA19} (see also \cite{AKADM} for 2d), the first author gave an alternative proof of scattering without the concentration-compactness, using Dodson-Murphy approach \cite{DM17}. We summarize results in the following statement.

\begin{theorem}[\cite{AKAR1}, \cite{AKA19}]\label{dichotomy}
Let $u_0\in H^1(\R^N)$, $0< s_c < 1$ and let $u(t)$ be the corresponding solution to \eqref{gH} with the maximal time existence interval $(-T_*, T^*)$. Suppose that $\mathcal{ME}[u_0]  < 1$.
\begin{enumerate}
\item 
If $\mathcal{G}[u_0] < 1$, then the solution exists globally in time with $\mathcal{G}[u(t)] < 1$ for all $t\in \R$, and $u(t)$ scatters in $H^1$.
\item 
If $\mathcal{G}[u_0] > 1$, then $\mathcal{G}[u(t)] > 1$ for all $t \in (-T_*, T^*)$. Moreover, if
\begin{enumerate}
\item 
$|x|u_0\in L^2(\R^N)$ or $u_0$ is radial, then the solution blows-up in finite time,
\item 
$u_0$ is of infinite variance and nonradial, then either the solution blows-up in finite time or there exits a sequence of times $t_n\rightarrow +\infty$ (or $t_n\rightarrow -\infty$) such that $\|\nabla u(t_n)\|_{L^2(\R^N)}\rightarrow \infty$.
\end{enumerate}
\end{enumerate}
\end{theorem}
We note that the proof of global existence in Theorem \ref{dichotomy} part (1) and blow-up in part 2(a) will work for $s_c=0$ and $s_c=1$. In the case $s_c=1$ (energy-critical gHartree), or equivalently $p = \frac{2N-b}{N-2}$ for $N\geq 3$, the inequality \eqref{GNC} becomes
\begin{equation}\label{eSob}
	Z(u)\leq C_{GN}\, \|\nabla u\|_{L^2(\R^N)}^{\frac{2(2N-b)}{N-2}},
\end{equation}
and the sharp constant $C_{GN}$ for \eqref{eSob} is given by (for instance, see \cite{DY19})
\begin{equation}\label{sharpcon}
C_{GN} =\pi^{b/2}\left(\frac{1}{N(N-2)\pi}\left(\frac{\Gamma(N)}{\Gamma(N/2)}\right)^{\frac{N-b+2}{2N-b}}\right)^{\frac{2N-b}{N-2}}\frac{\Gamma\left(\frac{N-b}{2}\right)}{\Gamma\left(N-\frac{b}{2}\right)}
= \big(C_N\big)^{\frac{2(2N-b)}{N-2}}\,C(N,b),
\end{equation}
where $C_N= \frac{1}{\sqrt{N(N-2)\pi}}\left(\frac{\Gamma(N)}{\Gamma(N/2)}\right)^{1/N} $ is the best constant for Sobolev inequality 
$$
\|u\|_{L^{\frac{2N}{N-2}}(\R^N)}\leq C_N\,\|\nabla u\|_{L^2(\R^N)}
$$
and
$
C(N,b)=\pi^{b/2}\,\frac{\Gamma\left(\frac{N-b}{2}\right)}{\Gamma\left(N-\frac{b}{2}\right)}\,\left(\frac{\Gamma(N)}{\Gamma(N/2)}\right)^{1-\frac{b}{N}}
$ 
is the sharp constant in Hardy-Littlewood-Sobolev inequality (see \cite{LiebL2001} ,\cite{Lieb83})
\begin{equation*}
\left|\int_{\R^N}\int_{\R^N} \frac{f(x)\, h(y)}{|x-y|^{b}} \,dx\,dy\right|\leq C(N,b)\,\|f\|_{L^{\frac{2N}{2N-b}}(\R^N)}\,\|h\|_{L^{\frac{2N}{2N-b}}(\R^N)}.
\end{equation*}
In our examples, we use  $b=N-2$, in which case one can verify that 
\begin{equation}\label{EQ}
Q(x) = \left(\frac{N(N-2)}{\pi^{N/2}}\,\Gamma\left(1+\frac{N}{2}\right)\right)^{\frac{N-2}{8}}\,\frac{1}{\left(1+|x|^2\right)^{\frac{N-2}{2}}} 
\end{equation}
is one of the solutions for
\begin{equation}\label{EQeq}
	\Delta Q + \left(|x|^{-(N-2)}\ast|Q|^\frac{2N-b}{N-2}\right)|Q|^{\frac{4-b}{N-2}}Q=0,
\end{equation}
where $(-\Delta)^{-1}f = I_{N-2}\ast f = \frac{\Gamma\left(\frac{N}{2}-1\right)}{4\,\pi^{N/2}}\,\frac{1}{|x|^{N-2}}\ast f$. In other words, the sharp constant $C_{GN}$ from \eqref{sharpcon} can be attained at $Q$, i.e., for $b=N-2$, we have an equality in \eqref{eSob} 
\begin{equation}\label{atQ}
	Z(Q) = C_{GN}\,\|\nabla Q\|^{\frac{2(N+2)}{N-2}}_{L^2(\R^N)}.
\end{equation}
Furthermore, for the function $Q$ in \eqref{EQ}, multiplying the equation \eqref{EQeq} by $Q$ and performing integration by parts, we have 
$
\|\nabla Q\|_{L^2(\R^N)}^2 = Z(Q). 
$ 
Thus, using this along with \eqref{atQ} we deduce that
\begin{equation}\label{EgradQ}
		\|\nabla Q\|^2_{L^2(\R^N)} = \frac{1}{\big(C_{GN}\big)^{(N-2)/4}}\quad\text{and}\quad E[Q] = \frac{2}{N+2}\,\|\nabla Q\|^2_{L^2(\R^N)} = \frac{2}{N+2}\,\frac{1}{\big(C_{GN}\big)^{(N-2)/4}}. 
\end{equation}
We next modify the definition of $\mathcal{ME}$ and $\mathcal{G}$ in \eqref{norMEG} and write  
\begin{equation*}
\mathcal{E}[u] = \frac{E[u]}{E[Q]}\quad\text{and}\quad \mathcal{G}[u(t)]=\frac{\|\nabla u(t)\|_{L^2(\R^N)}}{\|\nabla Q\|_{L^2(\R^N)}},
\end{equation*}
where the value of $\|\nabla Q\|_{L^2(\R^N)}$ is determined from \eqref{EgradQ} via the sharp constant defined in \eqref{sharpcon}. Now, as a consequence of the proof of Theorem \ref{dichotomy} in \cite{AKAR1}, global existence holds in $s_c=1$ case along with the blow-up in finite time for finite variance. We state the following analogous result (excluding scattering and blow-up for infinite variance, which will be considered elsewhere) for the energy-critical case.
\begin{theorem}\label{ec-dichotomy}
 Let $s_c=1$ and $u(t)$ be the solution of \eqref{gH} with $u_0\in\dot{H}^1(\R^N)$. Assume that $\mathcal{E}[u_0] < 1$.
 \begin{enumerate}
	\item If $\mathcal{G}[u_0] < 1$, then the solutions exists globally in time for all $t\in\R$.
		\item If $\mathcal{G}[u_0] > 1$ and either $u_0$ is radial or $xu_0\in L^2(\R^N)$, then $u(t)$ blows-up in finite time.  
	\end{enumerate}
\end{theorem}   
 \subsection{Gaussian initial data} We are now ready to consider examples, for which we take the Gaussian initial data of the form  
\begin{equation}\label{gid}
u_g(x,0) = \beta e^{-\frac{1}{2}\gamma|x|^2},\quad x\in\R^N,\,\,\,\beta,\gamma\in(0,\infty). 
\end{equation}
Then, the mass and initial variance of Gaussian data \eqref{gid} are
\begin{equation*}
M[u_g] = \beta^2\left(\frac{\pi}{\gamma}\right)^{N/2},\quad V(0) = \frac{\beta^2\,N\,\pi^{N/2}}{2\,\gamma^{\frac{N}{2}+1}}.
\end{equation*} 
For the convenience of the energy calculation we also record 
\begin{equation*}
\|\nabla u_g\|_{L^2(\R^N)}^2 = \frac{N\,\pi^{N/2}}{2}\,\frac{\beta^2}{\gamma^{(N-2)/2}}.
\end{equation*}
In what follows we consider mostly examples in 3d, with the convolution term 
$\frac1{|x|^{N-2}} \ast \bullet$ as it is the fundamental solution of the Laplacian. 

\subsubsection{Energy-subcritical case} Consider $p=3$ and $b=1$ in dimension $N=3$. In this case $s_c=\frac{1}{2}$, then \eqref{gH} takes the form
\begin{equation}\label{H12c}
	iu_t + \Delta u + \Big(|x|^{-1}\ast|u|^3\Big)|u|u = 0.
\end{equation}
The energy for \eqref{H12c} is
\begin{align*}
	E[u_g] = \frac{\pi^{3/2}}{4}\,\frac{\beta^2}{\gamma^{1/2}} \left(3-\frac{16\,\pi}{3^{7/2}}\,\frac{\beta^4}{\gamma^2}\right).
\end{align*}
The Pohozhaev identities take the form
\begin{equation*}
\|\nabla Q\|_{L^2(\R^3)}^2 = 2\|Q\|_{L^2(\R^3)}^2\quad\text{and}\quad Z(Q) = 3\|Q\|_{L^2(\R^3)}^2,
\end{equation*} 
which yields $E[Q] = \frac{1}{2} M[Q]$, where we computed (numerically) $M[Q] \approx 5.2339$ (for example, see \cite{YRZ3}).

We obtain the following thresholds, which are schematically represented in Figure \ref{F:1}:
\begin{figure}[ht]
\includegraphics[width=0.8\textwidth]{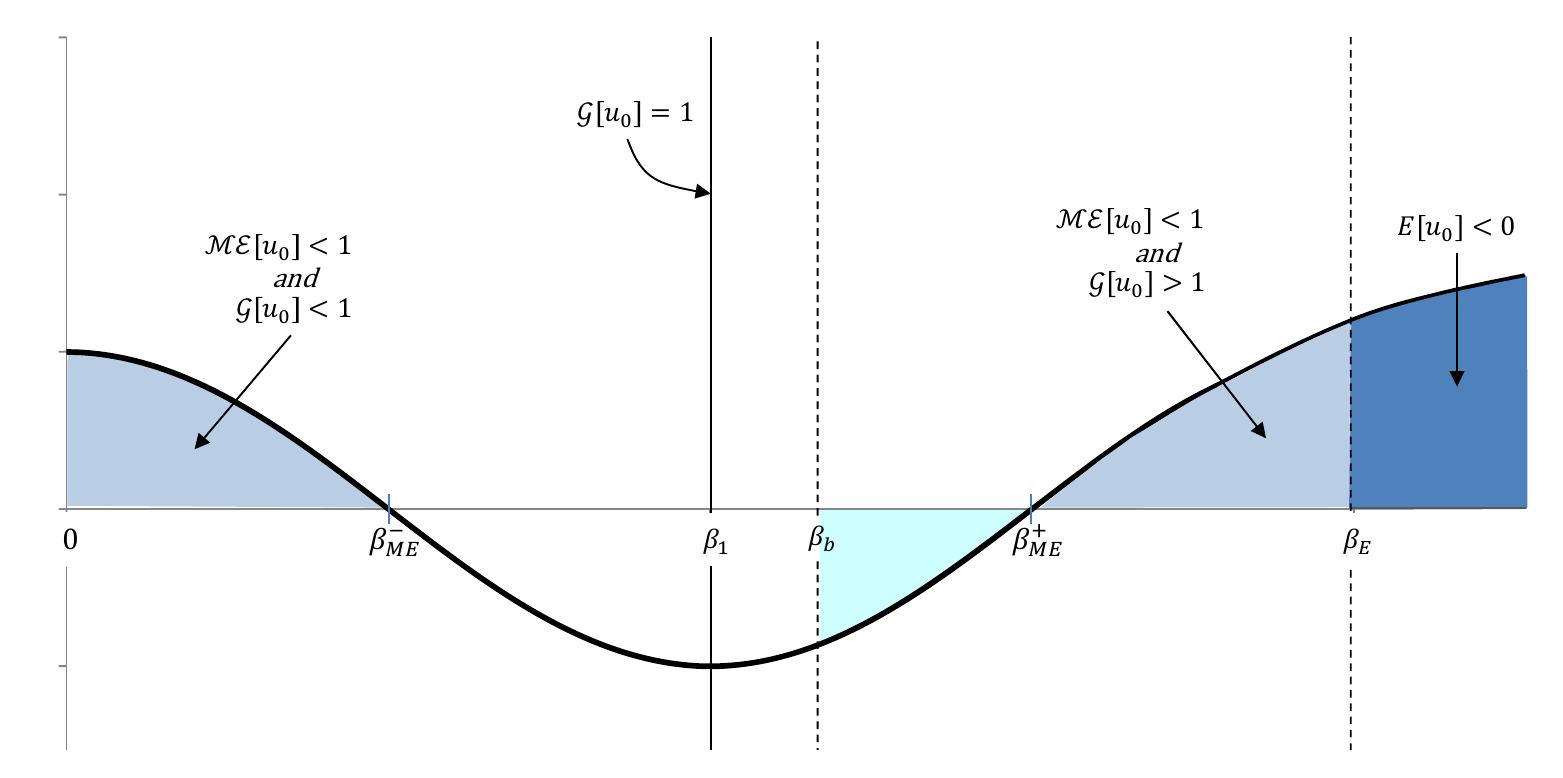}
\caption{Thresholds for the Gaussian data $u_0(x) = \beta \, e^{-|x|^2}$ in the energy-subcritical case, see \eqref{ex1a}-\eqref{ex1d}.} 
\label{F:1} 
\end{figure}
\begin{itemize}
\item blow-up with negative energy: $E[u_g] < 0$ if
	\begin{equation}\label{ex1a}
	\frac{\beta}{\sqrt{\gamma}} > \beta_E \equiv \frac{3^{9/8}}{2\,\pi^{1/4}} \approx 1.29, \qquad \qquad
	\end{equation} 
\item blow-up criteria Theorem \ref{blowupL} for positive energy: condition \eqref{L16real} gives
    \begin{equation}\label{ex1b}
		\frac{\beta}{\sqrt{\gamma}} > \beta_b \equiv \frac{3^{9/8}}{2^{5/4} \,\pi^{1/4}} \approx 1.08689,
	\end{equation}
\item  the mass-energy condition $\mathcal{ME}[u_g] < 1$ in Theorem \ref{dichotomy} yields
$$
\frac{\pi^3\,\beta^4}{4\,\gamma^2}\left(3 - \frac{16}{3^{7/2}}\,\frac{\beta^4}{\gamma^2}\right) < \frac{1}{2} \|Q\|_{L^2(\R^3)}^4,
$$ 
which implies
\begin{equation}\label{ex1c}
\frac{\beta}{\sqrt{\gamma}} < \beta_{ME}^- \approx 0.9586 \quad 
\text{and}\quad \frac{\beta}{\sqrt{\gamma}} > \beta_{ME}^+ \approx 1.1812.
\end{equation}
\item the mass-gradient condition $\mathcal{G}[u_g] = 1$ from Theorem \ref{dichotomy} (useful for the separation of the mass-energy conditions above in \eqref{ex1c}) gives
\begin{equation}\label{ex1d}
\frac{\beta}{\sqrt{\gamma}} < \beta_1 \equiv \frac{2^{1/2}}{3^{1/4}\,\pi^{3/4}}\,\|Q\|_{L^2(\R^3)} \approx 1.0418.
\end{equation}
\end{itemize}
We conclude from \eqref{ex1a}, \eqref{ex1b}, \eqref{ex1c} and \eqref{ex1d} that analytically proved ranges are: for scattering is below $\beta_{ME}^- \approx 0.9586$ and for blow-up is above $\beta_b \approx 1.08689$ (see Figure 1). 

\subsubsection{Energy-critical case} 
Consider $p=5$ and $b=1$ in the dimension $N=3$ and write the equation 
\begin{equation}\label{criticalH}
	iu_t + \Delta u + \Big(|x|^{-1}\ast|u|^5\Big)|u|^3u = 0, 
\end{equation}
which is energy-critical. The corresponding energy for \eqref{criticalH} is
\begin{equation*}
	E[u_g] = \frac{\pi^{3/2}}{4}\,\frac{\beta^2}{\gamma^{1/2}}\left(3 - \frac{16\,\pi}{5^{7/2}}\,\frac{\beta^8}{\gamma^2}\right).
\end{equation*} 
From \eqref{EQ} and \eqref{EQeq} we have that   
\begin{equation*}
	Q = \left(\frac{9}{4\,\pi}\right)^{\frac{1}{8}}\frac{1}{\sqrt{1+|x|^2}},
\end{equation*}
which solves 
\begin{equation*}
\Delta Q + \left(\frac{1}{|x|}\ast Q^5\right)Q^4 = 0,
\end{equation*}
where $\frac{1}{|x|}\ast f = 4\pi\,(-\Delta)^{-1}f$ in 3d. From \eqref{EgradQ} and \eqref{sharpcon}, we obtain 
\begin{equation*}
	\|\nabla Q\|_{L^2(\R^3)}^2 = \frac{3^{3/2}\,\pi^{7/4}}{2^{5/2}}\quad\text{and}\quad E[Q]=\frac{2}{5}\,\|\nabla Q\|_{L^2(\R^3)}^2 = \frac{3^{3/2}\,\pi^{7/4}}{2^{3/2}\,5}.
\end{equation*}
Then
\begin{itemize}
	\item the negative energy condition, $E[u_g]<0$ yields
	\begin{equation}\label{ex2a}
		\frac{\beta}{\gamma^{1/4}} > \beta_E \equiv \frac{5^{7/16}\,3^{1/8}}{2^{1/2}\,\pi^{1/8}}\approx 1.42161,
	\end{equation} 
	\item blow-up condition \eqref{L16real} (Theorem \ref{blowupL}) gives
	\begin{equation}\label{ex2b}
		\frac{\beta}{\gamma^{1/4}} > \beta_b \equiv \frac{5^{5/16}\,3^{1/8}}{2^{1/2}\,\pi^{1/8}} \approx 1.16254,
	\end{equation}
	\item the energy condition $\mathcal{E}[u_g] < 1$ in Theorem \ref{ec-dichotomy} yields
	$$
	\frac{\pi^{3/2}}{4}\,\frac{\beta^2}{\gamma^{1/2}}\left(3 - \frac{16\,\pi}{5^{7/2}}\,\frac{\beta^8}{\gamma^2}\right) < \frac{3^{3/2}\,\pi^{7/4}}{2^{3/2}\,5},
	$$ 
	which implies
	\begin{equation}\label{ex2c}
		\frac{\beta}{\gamma^{1/4}} < \beta_E^- \approx 0.812225 \quad \text{and}\quad \frac{\beta}{\gamma^{1/4}} > \beta_E^+ \approx 1.34423,
	\end{equation}
	and the gradient condition for global existence $\mathcal{G}[u_g] < 1$ gives
	\begin{equation}\label{ex2d}
		\frac{\beta}{\gamma^{1/4}} < \beta_1 \approx 0.902925.
	\end{equation}
\end{itemize}
We conclude from \eqref{ex2a}, \eqref{ex2b}, \eqref{ex2c} and \eqref{ex2d} that analytically proved ranges are: for global existence is below $\beta_E^- \approx 0.812225$ and for blow-up is above $\beta_b \approx 1.16254$ (see Figure 2).

\begin{figure}[ht]
\includegraphics[width=0.8\textwidth]{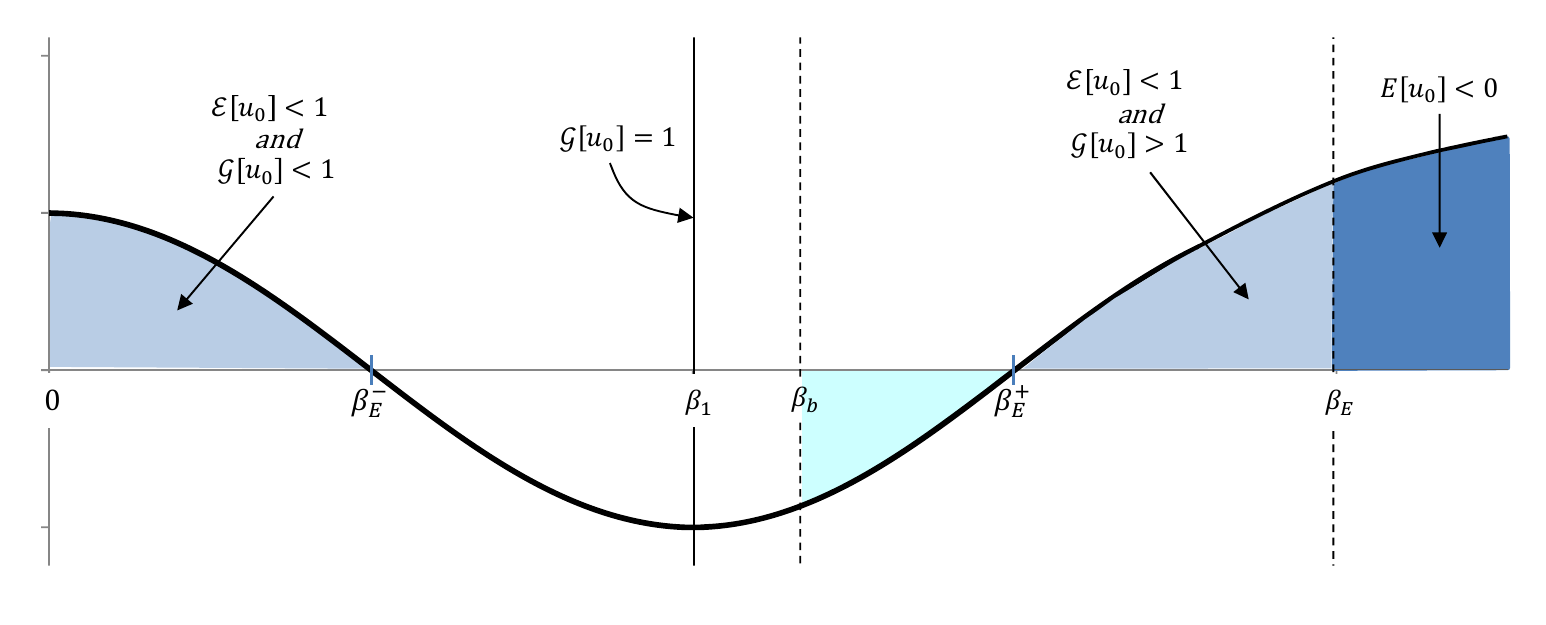}
\caption{Thresholds for Gaussian data $u_0 = \beta \, e^{-|x|^2}$ in the energy-critical case,see \eqref{ex2a}-\eqref{ex2d} and \eqref{ex2ia}-\eqref{ex2id}.} 
\label{F:2}
\end{figure}

For convenience, we provide one more energy-critical example in $4$d, 
\begin{equation}\label{criticalH1}
	iu_t + \Delta u + \Big(|x|^{-2}\ast|u|^3\Big)|u|u = 0.
\end{equation} 
The energy for \eqref{criticalH1} is given by
\begin{equation*}
	E[u_g] = \pi^2\,\frac{\beta^2}{\gamma} \left(1-\frac{\pi^2}{81}\,\frac{\beta^4}{\gamma^2}\right).
\end{equation*} 
Again, from \eqref{EQ} and \eqref{EQeq}, we have that  
$$
Q =  \frac{2}{\sqrt{\pi}}\,\frac{1}{1+|x|^2},
$$
which solves
\begin{equation*}
	\Delta Q + \left(\frac{1}{|x|^2}\ast Q^3\right)Q^2 = 0,
\end{equation*}
where $\frac{1}{|x|^2}\ast f = 4\pi^2(-\Delta)^{-1}f$ in $4$d. We compute from \eqref{EgradQ} and \eqref{sharpcon}
$$
\|\nabla Q\|_{L^2(\R^4)}^2 = \frac{16\,\pi}{3}\quad\text{and}\quad E[Q]=\frac{1}{3}\|\nabla Q\|_{L^2(\R^4)}^2 = \frac{16\,\pi}{9}.
$$
Then
\begin{itemize}
	\item the negative energy condition corresponds to
	\begin{equation}\label{ex2ia}
		\frac{\beta}{\sqrt{\gamma}} > \beta_E \approx 1.69257,
	\end{equation}
	\item blow-up occurs (according to Theorem \ref{blowupL} condition \eqref{L16real}) when 
	\begin{equation}\label{ex2ib}
			\frac{\beta}{\sqrt{\gamma}} > \beta_b \approx 1.28607,
	\end{equation} 
	\item the energy condition $\mathcal{E}[u_g] < 1$ in Theorem \ref{ec-dichotomy} gives
	$$
	\pi^2\,\frac{\beta^2}{\gamma}\left(1 - \frac{\pi^2}{81}\,\frac{\beta^4}{\gamma^2}\right) < \frac{16\,\pi}{9},
	$$ 
	which implies
	\begin{equation}\label{ex2ic}
		\frac{\beta}{\sqrt{\gamma}} < \beta_E^- \approx 0.768792 \quad\text{and}\quad 	\frac{\beta}{\sqrt{\gamma}} > \beta_E^+ \approx 1.58845,
	\end{equation}
	and the gradient condition for global existence $\mathcal{G}[u_g] < 1$ gives
	\begin{equation}\label{ex2id}
		\frac{\beta}{\sqrt{\gamma}} < \beta_1 \approx  0.921318.
	\end{equation}
\end{itemize}
We conclude from \eqref{ex2ia}, \eqref{ex2ib}, \eqref{ex2ic} and \eqref{ex2id} that analytically proved ranges are: for global existence is below $\beta_E^- \approx 0.768792$ and for blow-up is above $\beta_b \approx 1.28607$ (see Figure 2).

\subsubsection{Energy-supercritical case} Finally, we consider $p=7$ and $b=1$ in the dimension $N=3$
\begin{equation}\label{Esuperc}
	iu_t + \Delta u + \Big(|x|^{-1}\ast|u|^7\Big)|u|^5u = 0.
\end{equation}
 In this case $s_c=\frac{7}{6}>1$, thus, the energy-supercritical regime. The energy for \eqref{Esuperc} is given by
 \begin{equation*}
 	E[u_g] = \frac{\pi^{3/2}}{4}\,\frac{\beta^2}{\gamma^{1/2}} \left(3-\frac{16\,\pi}{7^{7/2}}\,\frac{\beta^{12}}{\gamma^2}\right).
 \end{equation*}
Then, in the energy-supercritical case we have
\begin{itemize}
	\item $E[u_g] < 0$ if
	\begin{equation}\label{ex3a}
	\frac{\beta}{\gamma^{1/6}} > \beta_E \equiv \frac{3^{1/12}\,\,7^{7/24}}{2^{1/3}\,\,\pi^{1/12}}\approx 1.3946799,
	\end{equation} 
	\item condition \eqref{L16real} (Theorem \ref{blowupL}) gives
	\begin{equation}\label{ex3b}
	\frac{\beta}{\gamma^{1/6}} > \beta_b \equiv \frac{3^{1/12}\,\,7^{7/24}}{2^{7/12}\,\,\pi^{1/12}}\approx 1.17278.
	\end{equation}
\end{itemize}
Note that except for the two conditions above no other information about scattering or blow up thresholds is known in the energy-supercritical case (except for the small data shown earlier in this paper).
\begin{figure}[ht]
\includegraphics[width=0.8\textwidth]{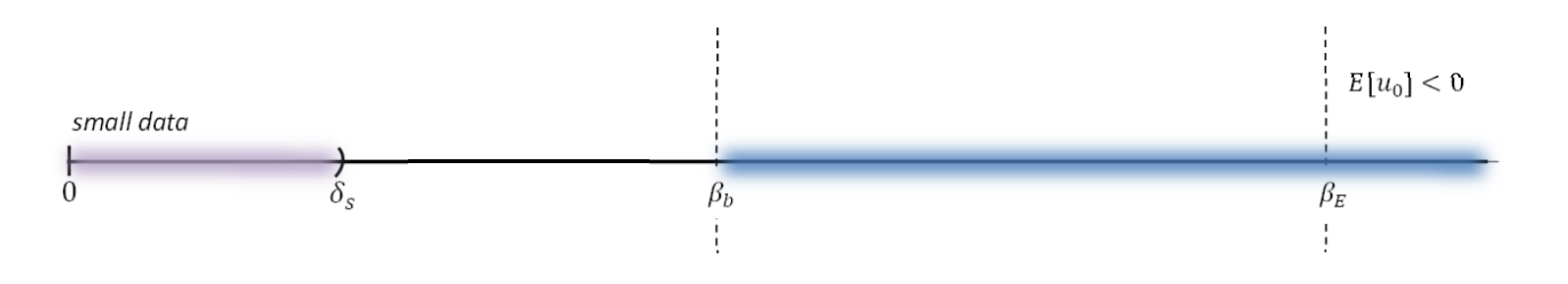}
\caption{Thresholds for Gaussian data $u_0 = \beta \, e^{-|x|^2}$ in the energy-supercritical case, see \eqref{ex3a}-\eqref{ex3b}.} 
\label{F:3}
\end{figure}

\bibliography{Andy-references}
\bibliographystyle{abbrv}

\end{document}